\documentclass[twocolumn,amsthm]{autart}    

\usepackage[utf8]{inputenc}             \usepackage[T1]{fontenc}                \usepackage{url}                        \usepackage{booktabs}                   \usepackage{nicefrac}                   

\usepackage{tikz}       

\usepackage{mathtools}

\usepackage{xparse}                 

\usepackage{amsfonts}                   

\usepackage[bookmarks=true]{hyperref}   

\makeatletter \let\cl@part\relax \makeatother
\usepackage[capitalise]{cleveref}

\usepackage{csquotes}

 \def\diffd{\mathrm{d}}

\DeclareDocumentCommand\differential{ o g d() }{ \IfNoValueTF{#2}{
		\IfNoValueTF{#3}
			{\diffd\IfNoValueTF{#1}{}{^{#1}}}
			{\mathinner{\diffd\IfNoValueTF{#1}{}{^{#1}}\argopen(#3\argclose)}}
		}
		{\mathinner{\diffd\IfNoValueTF{#1}{}{^{#1}}#2} \IfNoValueTF{#3}{}{(#3)}}
	}
\DeclareDocumentCommand\dd{}{\differential} 

\newcommand{\T}{^\mathsf{T}}

\DeclareMathOperator{\Tr}{\mathbf{tr}}
\DeclareMathOperator{\interior}{int}
\DeclareMathOperator{\E}{\mathbb{E}}

\DeclarePairedDelimiter\abs{\lvert}{\rvert}

\newcommand*{\XSet}{\mathcal{X}}
\newcommand*{\USet}{\mathcal{U}}

\newcommand{\lft}{\mathopen{}\mathclose\bgroup\left}
\newcommand{\rht}{\aftergroup\egroup\right} 
\newcommand*\circled[1]{\tikz[baseline=(char.base)]{
    \node[shape=circle,draw,inner sep=1pt] (char) {#1};}}

\begin{document}
\begin{frontmatter}

\title{Almost-Sure Safety Guarantees of Stochastic Zero-Control Barrier Functions Do Not Hold\thanksref{footnoteinfo}} 

\thanks[footnoteinfo]{Corresponding author O.~So}

\author[MIT]{Oswin So}\ead{oswinso@mit.edu},    \author[WashU]{Andrew Clark}\ead{andrewclark@wustl.edu},               \author[MIT]{Chuchu Fan}\ead{chuchu@mit.edu}  

\address[MIT]{Massachusetts Institute of Technology, United States}  \address[WashU]{Washington University in St. Louis, United States}             


\begin{abstract}                          The 2021 paper ``Control barrier functions for stochastic systems'' provides theorems that give almost sure safety guarantees given stochastic zero control barrier function (ZCBF). Unfortunately, both the theorem and its proof is invalid. In this letter, we illustrate on a toy example that the almost sure safety guarantees for stochastic ZCBF do not hold and explain why the proof is flawed.
Although stochastic reciprocal barrier functions (RCBF) also uses the same proof technique, we provide a different proof technique that verifies that stochastic RCBFs are indeed safe with probability one.
Using the RCBF, we derive a modified ZCBF condition that guarantees safety with probability one.
Finally, we provide some discussion on the role of unbounded controls in the almost-sure safety guarantees of RCBFs, and show that the rate of divergence of the ratio of the drift and diffusion is the key for whether a system has almost sure safety guarantees.

\end{abstract}

\end{frontmatter}
\endNoHyper

\section{Introduction}
The paper \cite{clark2021control} presented a framework for generalizing Control Barrier Functions (CBF) from deterministic systems to \textit{stochastic systems} that guarantees safety with probability $1$. In contrast to many works examining safety for stochastic systems in continuous time that give results for \textit{finite-time} safety \cite{prajna2004stochastic,prajna2007framework,wang2021safety,black2023safety,abel2023prescribed}, the paper \cite{clark2021control} stands out in that it provides a way to guarantee safety \textit{with probability $1$} for \textit{all time}. Analogous to the deterministic case \cite{ames2016control}, the paper \cite{clark2021control} constructs stochastic versions of the reciprocal CBF (RCBF) and the zero-CBF (ZCBF). The almost-sure safety guarantees of the stochastic ZCBF has been used in recent works (\cite{pereira2021safe,song2022generalization,vahs2023non,vahs2023risk,enwerem2023safe}). However, as we show in this letter, these almost-sure safety guarantees are incorrect. Interestingly, the same SZCBF has been used in \cite{wang2021safety}, where only finite-time safety guarantees are invoked.

In this letter, we reveal that both the theorem and the proof of almost-sure safety of stochastic zero-CBF (ZCBF) in \cite[Theorem 3]{clark2021control} are incorrect via a simple counterexample.
Next, we examine the implications of this error on the almost-sure safety of stochastic Reciprocal CBFs, which use the same proof technique. In particular, the same counterexamples for ZCBFs are not valid here, and it remains to be seen whether the almost-sure safety of stochastic RCBFs can be proven using the original proof technique.

\section{Theorem Statement and Illustration on Simple Examples}
We consider the following stochastic control system with respect to a filtered probability space $(\Omega, \mathcal{F}, \{\mathcal{F}_t\}_{t \geq 0}, \mathbb{P})$ satisfying the usual conditions \cite{karatzas2012brownian}, described by the following stochastic differential equation (SDE) with control-affine drift
\begin{equation} \label{eq:sde}
    \dd{x}_t = \Big( f(x_t) + g(x_t) u_t \Big) \dd{t} + \sigma(x_t) \dd{W}_t
\end{equation}
for states $x_t \in \XSet \subseteq \mathbb{R}^{n_x}$ and $u_t \in \USet \subseteq \mathbb{R}^{n_u}$.

\begin{defn}[Zero-CBF]
    The function $h : \XSet \to \mathbb{R}$ is a zero-CBF for a system described by the SDE \eqref{eq:sde} if for all $x \in \XSet$ satisfying $h(x) > 0$, there exists a $u \in \USet$ satisfying
    \begin{equation} \label{eq:zcbf}
        \frac{\partial h}{\partial x} \Big( f(x) + g(x) u\Big) + \frac{1}{2} \Tr\left(\sigma\T \frac{\partial^2 h}{\partial x^2} \sigma\right) \geq -h(x).
    \end{equation}
\end{defn}

\begin{thm}[{{\cite[Theorem 3]{clark2021control}}}] \label{thm:zcbf_wp1}
Suppose that there exists an ZCBF $h$ for a controlled stochastic process $x_t$ described by \eqref{eq:sde}, and at each time $t$, $u_t$ satisfies \eqref{eq:zcbf}. Then $\Pr(x_t \in C, \forall t\geq0) = 1$, provided that $x_0 \in C$.
\end{thm}

Before we examine the proof of Theorem \ref{thm:zcbf_wp1}, we first illustrate that the
result does not hold
via the following simple counterexample.
\begin{exmp}[Uncontrolled Brownian Motion]
    Consider the case of (uncontrolled) Brownian Motion by taking 
    \begin{equation}
        f(x) = 0, \quad g(x) = 0, \quad \sigma(x) = 1.
    \end{equation}
    This reduces the SDE \eqref{eq:sde} to $\dd{x}_t = \dd{W}_t$, which has the solution $x_t = W_t$. We now construct the ZCBF $h$ as $h(x) = x$, where the safe operating region $\mathcal{C}$ corresponds to the non-negative reals $\mathcal{C} = \{ x : x \geq 0 \} = \mathbb{R}_{\geq 0}$.

    It can be readily verified that $h$ is a ZCBF as it satisfies \eqref{eq:zcbf}, which in this case reduces to
    \begin{equation}
        h(x) > 0 \implies 0 \geq -h(x).
    \end{equation}
    However, Theorem \ref{thm:zcbf_wp1} claims that
    \begin{equation}
      \Pr(W_t \geq 0, \forall t \geq 0) = 1  
    \end{equation}
    which is not true, since $W_t$ is normally distributed with zero mean and variance $t$.
\end{exmp}

\section{Flaw of the Proof}
The main flaw in the proof comes from an implicit assumption when applying mathematical induction. Let $\theta \in (0, h(x_0)]$ and define the following sequence of stopping times $\eta_i$ and $\zeta_i$ for $i = 0, 1, \dots$ as
\begin{align}
    \eta_0 &= 0, \\
    \zeta_0 &= \inf\{ t : h(x_t) > \theta \} \\
    \eta_i &= \inf\{ t : h(x_t) < \theta, t > \zeta_{i-1} \}, \quad i = 1, 2, \dots \\
    \zeta_i &= \inf\{ t : h(x_t) > \theta, t > \eta_{i-1} \}, \quad i = 1, 2, \dots.
\end{align}
Next, define the random process $U_t$ as follows:
\begin{equation}
    U_t = \theta + \sum_{i=0}^\infty \left[ \int_{\eta_i \wedge t}^{\zeta_i \wedge t} -\theta \dd{\tau} + \int_{\eta_i \wedge t}^{\zeta_i \wedge t} \sigma \frac{\partial h}{\partial x} \dd{W}_\tau \right].
\end{equation} 
The flaw then comes from the following statement:
\begin{displayquote}
We will first prove by induction that $h(x_t) \geq U_t$ and $U_t \leq \theta$.
\end{displayquote}
While the intention is to prove that this property holds for all $t \geq 0$, this turns out to not be the case. In particular, the mathematical induction is performed by showing that this property holds for all closed intervals $[\eta_i, \zeta_i]$ and $[\zeta_i, \eta_{i+1}]$. However, unless the union of these intervals covers $[0, \infty)$, i.e.,
\begin{equation}
  \lim_{i \to \infty} \eta_i = \lim_{i \to \infty} \zeta_i = \infty,
\end{equation}
the property does not hold for all $t \geq 0$, and the proof is invalid.

\subsection{Proof Counterexample for Uncontrolled Brownian Motion}
In the case of (uncontrolled) Brownian motion, it can be shown that, except for $\eta_0 = 0$, all the stopping times take on the same value.
\begin{lem}[]
    Suppose $\dd{x_t} = \dd{W}_t$ and $h(x_t) = x_t = W_t$. Then, the stopping times $\eta_i, i = 1, 2, \dots$ and $\zeta_i, i = 0, 1, 2, \dots$ are all equal almost surely, i.e.,
    \begin{equation}
        \zeta_0 = \eta_1 = \zeta_1 = \eta_2 = \dots,\quad a.s.
    \end{equation}
\end{lem}
\begin{proof}

    Note that $\{ \eta_1 = \zeta_0 \} \in \mathcal{F}_{\zeta_0}^+$. Thus, by Blumenthal's 0-1 law \cite{le2016brownian}, $\Pr(\eta_1 = \zeta_0) \in \{0, 1\}$. By symmetry of Brownian motion, for any $\tau > 0$,
    \begin{equation}
        \Pr( \eta_1 \leq \zeta_0 + \tau) \geq \Pr(W_{\zeta_0 + \tau} < \theta ) = \Pr(W_{\zeta_0 + \tau} > \theta) = \frac{1}{2}
    \end{equation}
    Taking $\tau \downarrow 0$, we get that $\Pr(\eta_1 = \zeta_0) \geq \frac{1}{2}$. Hence, $\eta_1 = \zeta_0$ almost surely. Repeating this argument then shows that all the stopping times are equal almost surely.
\end{proof}
Moreover, $\zeta_0$ is finite almost surely. Hence, $\lim_{i \to \infty} \zeta_i = \eta_i < \infty$ almost surely.

\section{Stochastic Reciprocal Control Barrier Functions are Safe With Probability One}
Analogous to the deterministic case, the paper \cite{clark2021control} also defines a stochastic \textit{reciprocal} control barrier function (RCBF) $B$, where $B$ tends to infinity as the system state approaches the boundary of the safe region $\mathcal{C}$:
\begin{defn}[Reciprocal-CBF]
    A reciprocal CBF is a function $B : \XSet \to \mathbb{R}$ that is locally Lipschitz, twice differentiable on $\interior(C)$, and satisfies the following properties:
    \begin{enumerate}
        \item There exist class-$\mathcal{\kappa}$ functions
        \footnote{
        A function $\alpha : \mathbb{R} \to \mathbb{R}$ is class-$\mathcal{\kappa}$ if it is strictly increasing and $\alpha(0) = 0$
        }
        $\alpha_1$ and $\alpha_2$ such that
        \begin{equation} \label{eq:rcbf_bounds}
            \frac{1}{\alpha_1( h(x) )} \leq B(x) \leq \frac{1}{\alpha_2( h(x) )},
        \end{equation}
        for all $x \in \interior(C)$.
\item There exists a class-$\mathcal{\kappa}$ function $\alpha_3$ such that, for all $x \in \interior(C)$, there exists $u \in \USet$ such that
        \begin{equation} \label{eq:rcbf}
            \frac{\partial B}{\partial x}( f(x) + g(x) u ) + \frac{1}{2} \Tr\lft(\sigma\T \frac{\partial^2 B}{\partial x^2} \sigma\rht) \leq \alpha_3( h(x) ).
        \end{equation}
    \end{enumerate}
\end{defn}
Accordingly, a similar theorem of almost-sure safety is stated for stochastic RCBFs.
\begin{thm}[see {{\cite[Theorem 1]{clark2019control}, \cite[Theorem 2]{clark2021control}}}] \label{thm:rcbf_wp1}
Suppose that there exists an RCBF $B$ for a controlled stochastic process $x_t$ described by \eqref{eq:sde}, and at each time $t$, $u_t$ satisfies \eqref{eq:rcbf}. Then $\Pr(x_t \in C, \forall t\geq0) = 1$, provided that $x_0 \in C$.
\end{thm}

The proof of Theorem \ref{thm:rcbf_wp1} from \cite{clark2019control} also makes use of the same mathematical induction argument as a proof technique by defining a sequence of stopping times $\eta_i$ and $\zeta_i$ and showing that an analogous property holds over all closed intervals $[\eta_i, \zeta_i]$ and $[\zeta_i, \eta_{i+1}]$. Similar to the ZCBF case, the validity of the proof hinges on whether it can be shown that the union of these intervals covers $[0, \infty)$ or not. While we were unable to show that this is true using this proof technique, we are also unable to find a counter-example.

Instead, we use a different proof technique. Similar to before, we construct a semimartingale $\tilde{B}$ with linear drift that upper-bounds the RCBF $B$. However, instead of using stopping times, we make use of local times and Tanaka's formula.

To start, we begin with Tanaka's formula\footnote{Some sources define local time with an extra coefficient of $1/2$ \cite{karatzas2012brownian}, which results in the coefficient for $L_t^a$ being different.}.
\begin{thm}{(Tanaka's Formula \cite[p.~222]{revuz2013continuous}, \cite[p.~237]{le2016brownian})}{\label[theorem]{thm:tanaka}}
    Let $X$ be a continuous semimartingale and $a \in \mathbb{R}$. There exists an increasing process $(\, L_t^a(X) \,)_{t\geq0}$ called the local time of $X$ in $a$, such that 
    \begin{equation}
        \begin{split}
        &\mathrel{\phantom{=}}\max(X_t - a, 0) \\
        &= \max(X_0 - a, 0) + \int_0^t \mathbb{I}_{(X_s > a)} \dd{X}_s + \frac{1}{2} L_t^a(X).
        \end{split}
    \end{equation}
\end{thm}
Next, we state Doob's Martingale Inequality\footnote{This result is also referred to as Doob's Maximal Inequality \cite{revuz2013continuous,le2016brownian}.} \cite{karatzas2012brownian} below.
\begin{lem}{(Doob's Martingale Inequality \cite[Thm 3.8]{karatzas2012brownian})}{}
    Let ($X_t, \mathcal{F}_t$) be a submartingale, $[t_0, t_1]$ a subinterval of $[0, \infty)$, and $\lambda > 0$. Then,
    \begin{equation} \label{eq:doob_ineq}
        \lambda \Pr\left( \sup_{t_0 \leq t \leq t_1} X_t \geq \lambda\right) \leq \E[\max\{X_t, 0\}]
    \end{equation}
\end{lem}

We now prove that the existence of a stochastic RCBF implies that the system remains safe with probability one. To do so, we first prove the following lemma, which probabilistically bounds $B(x)$ with another stochastic process with a linear drift term.
\begin{lem}{}{\label[lemma]{lem:B_tilde_bound}}
    For any $t > 0$ and $\delta > 0$, let $M \in \mathbb{R}$ be such that $P\left( \frac{1}{2}L_t^{B_0}(B) > M \right) < \frac{\delta}{2}$, and define the stochastic process $\tilde{B}$ as
    \begin{align} 
        \tilde{b} &\coloneqq \alpha_3\left( \alpha_2^{-1} \left( \frac{1}{B_0} \right) \right), \label{eq:b_drift_def} \\
        \tilde{B}_t &\coloneqq B_0 + M + \tilde{b} t
            + \int_0^t \frac{\partial B}{\partial x} \sigma(x_s) \mathbb{I}_{(X_s > a)} \dd{W}_s. \label{eq:B_tilde_def}
    \end{align}
    Then, with probability at least $(1 - \delta/2)$, we have that $B_s \leq \tilde{B}_s$ for all $s \in [0, t]$.
\end{lem}
\begin{proof}
    Splitting the drift and diffusion terms depending on if $B_\tau \geq B_0$, we have that
    \begin{align}
        &\mathrel{\phantom{=}} B_s \notag \\
        &\begin{aligned}[t]
        &= B_0 + \int_0^s \underbrace{\frac{\partial B}{\partial x} a(x_\tau)
        + \frac{1}{2}\Tr\left(\sigma(x_\tau)\T \frac{\partial^2 B}{\partial x^2} \sigma(x_\tau) \right)}_{\coloneqq \overline{a}_\tau} \dd{\tau} \\
        &\hphantom{= B_0} + \int_0^s \underbrace{\frac{\partial B}{\partial x}\sigma(x_\tau)}_{\coloneqq \overline{b}_\tau} \dd{W}_\tau
        \end{aligned} \\
&= B_0 + \underbrace{\int_0^s \overline{a}_\tau \mathbb{I}_{(B_\tau \geq B_0)} \dd{\tau} + \int_0^s \overline{b}_\tau \mathbb{I}_{(B_\tau \geq B_0)} \dd{W}_\tau}_{\circled{\footnotesize 1}} \notag \\
        &\hphantom{= B_0} + \underbrace{\int_0^s \overline{a}_\tau \mathbb{I}_{(B_\tau < B_0)} \dd{\tau}}_{\circled{\footnotesize 2}} + \int_0^s \overline{b}_\tau \mathbb{I}_{(B_\tau < B_0)} \dd{W}_\tau
    \end{align}
Applying Tanaka's formula \Cref{thm:tanaka} to $\min\{ B_s - B_0, 0 \}$ gives us
\begin{equation}
    \begin{split}
    &\mathrel{\phantom{=}} \min\{ B_s - B_0, 0 \} \\
&= \int_0^s \mathbb{I}_{(B_\tau < B_0)} \overline{a}_\tau \dd{\tau} + \int_0^s \mathbb{I}_{(B_\tau < B_0)} \overline{b}_\tau \dd{W}_\tau \\
    &\qquad - \frac{1}{2} L_s^{B_0}(B).
    \end{split}
    \raisetag{1.2\baselineskip}
\end{equation}
Since the local time $L^{B_0}(X)$ is an increasing process and $s \leq t$:
\begin{align}
&\mathrel{\phantom{=}} \int_0^s \mathbb{I}_{(B_\tau < B_0)} \overline{a}_\tau \dd{\tau} + \int_0^s \mathbb{I}_{(B_\tau < B_0)} \overline{b}_\tau \dd{W}_\tau \notag \\
    &= \min\{ B_s - B_0, 0 \} + \frac{1}{2} L_s^{B_0}(B) \\
    &\leq \frac{1}{2} L_t^{B_0}(B) \label{eq:local_time_bound}
\end{align}
Moreover, by the definition of the RCBF \eqref{eq:rcbf_bounds}--\eqref{eq:rcbf} and \eqref{eq:b_drift_def}, for any $s \geq 0$,
\begin{equation} \label{eq:rcbf_drift_bound}
\int_0^s \overline{\alpha}_\tau \mathbb{I}_{(B_\tau < B_0)} \dd{\tau} \leq \tilde{b} s.
\end{equation}
Hence, using \eqref{eq:local_time_bound} to bound $\circled{\footnotesize 1}$ and \eqref{eq:rcbf_drift_bound} to bound $\circled{\footnotesize 2}$, we have that
\begin{align}
B_s &\leq B_0 + \tilde{b} s + \int_0^s \overline{b}_\tau \mathbb{I}_{(B_\tau \geq B_0)} \dd{W}_\tau + \frac{1}{2} L_t^{B_0}(B).
\end{align}
By definition of $M$, $\Pr(\frac{1}{2} L_t^{B_0}(B) \leq M) \geq (1 - \delta / 2)$.
Hence, with probability at least $(1 - \delta/2)$,
\begin{equation}
    B_s \leq \tilde{B}_s, \quad \forall s \in [0, t].
\end{equation}
\end{proof}

We now proceed to state and prove the main theorem.
\begin{thm}{}\label[theorem]{thm:rcbf_wp1:fixed}
    Suppose there exists a stochastic RCBF for the stochastic process $X$. Then, for all $t \geq 0$, $P(X_t \in C) = 1$, provided that $X_0 \in \interior(C)$.
\end{thm}
\begin{proof}
    Let $B$ be a stochastic RCBF, and define $B_t = B(x_t)$.
    Since each sample path of $X_t$ is continuous, each sample path of $B_t$ is also continuous.
    Hence, if $x_t \not \in C$ for some $t$, then there exists some $s < t$ such that $h(x_s ) = 0$ and hence $B_s = \infty$. The goal is to then show that, for all $t > 0$ and $\delta \in (0, 1)$,
    \begin{equation} \label{eq:p_B_infty_delta}
        \Pr(\sup_{s < t} B_s = \infty) < \delta,
    \end{equation}
    which implies that $\Pr(X_t \in C) = 1$ for all $t\geq 0$.

    Let $t>0$ and $\delta >0$. To show \Cref{eq:p_B_infty_delta} holds, it is sufficient to show the existance of a $K > 0$ such that $\Pr(\sup_{s < t} B_s > K) < \delta$. Define $\tilde{B}$ as in \eqref{eq:B_tilde_def}.
    Then, by \Cref{lem:B_tilde_bound}, $B_s \leq \tilde{B}_s$ for all $s\in[0,t]$ with probability at least $(1 - \delta/2)$.
    Furthermore, since $\tilde{B}$ is a submartingale, applying Doob's martingale inequality \eqref{eq:doob_ineq} gives us that
    \begin{equation}
        \Pr\left( \sup_{s \in [0, t]} \tilde{B}_s > K \right) \leq \frac{\delta}{2}.
    \end{equation}
    Combining these two results and applying Boole's inequality thus gives us
    \begin{align}
        &\mathrel{\phantom{=}} \Pr\left(\sup_{s \in [0, t]} B_s \leq K \right) \notag\\
        &\geq \Pr\left( B_s \leq \tilde{B}_s\, \forall s \in [0, t],\; \sup_{s \in [0,t]} \tilde{B}_s \leq K \right) \\
        &= 1 - \left( \Pr(\exists s \in [0, t],\; B_s > \tilde{B}_s) + \Pr\left(\sup_{s \in [0, t]} \tilde{B}_s > K \right) \right) \\
        &= 1 - \delta.
    \end{align}
    In other words, we have thus shown that
    \begin{equation}
        \Pr\left( \sup_{s \in [0, t]} B_s > K \right) \leq \delta.
    \end{equation}
\end{proof}

\section{Modified ZCBF conditions for safety with probability one}
Drawing on the valid safety guarantees of RCBFs, we can derive a modified ZCBF condition for safety with probabilty one by examining the conditions on $h$ that guarantee that $B$ is an RCBF.

Denote by $\overline{\mu}$ and $\overline{\sigma}$ the drift and diffusion terms of $h_t$ as in \eqref{eq:h_ito}, i.e.,
\begin{equation} \label{eq:h_ito}
    \begin{split}
    \dd{h}_t &= \underbrace{\frac{\partial h}{\partial x} ( f(x_t) + g(x_t) u ) \dd{t} + \frac{1}{2} \Tr\left(\sigma\T \frac{\partial^2 h}{\partial x_t^2} \sigma\right)}_{\coloneqq \tilde{\mu}_t} \dd{t} \\
    &\qquad + \underbrace{\frac{\partial h}{\partial x} \sigma(x_t)}_{\coloneqq \tilde{\sigma}_t} \dd{W}_t.
    \end{split}
\end{equation}
Then, by taking $B(x) = 1/h(x)$, we can derive the following corollary to \Cref{thm:rcbf_wp1:fixed}.
\begin{cor}[Modified ZCBF Conditions for Safety with Probability One]\label[theorem]{thm:zcbf_like}
    Suppose there exists a function $h: \XSet \to \mathbb{R}$ and a class-$\mathcal{\kappa}$ function $\alpha_3$ where, for all $x \in \XSet$ satisfying $h(x) > 0$, there exists $u \in \USet$ such that
    \begin{equation} \label{eq:zcbf_like_ineq}
        \tilde{\mu} - \frac{\tilde{\sigma}^2}{h(x)} \geq - h(x)^2 \alpha_3(h).
    \end{equation}
    Then, for all $t \geq 0$, $\Pr(x_t \in C) = 1$, provided that $x_0 \in \interior(C)$.
\end{cor}
\begin{proof}
    Applying Ito's lemma to $B(x) = 1/h(x)$, we obtain the drift term $\hat{\mu}_t$ of $B$ as
    \begin{align}
        \hat{\mu}_t 
        &= \frac{\partial B}{\partial h} \tilde{\mu}_t + \frac{1}{2} \tilde{\sigma}_t^2 \frac{\partial^2 B}{\partial h^2} \\
        &= -h^{-2} \tilde{\mu}_t + h^{-3} \tilde{\sigma}_t^2 \label{eq:B_drift_h}
    \end{align}
    Applying the RCBF conditions \eqref{eq:rcbf_drift_bound} to \eqref{eq:B_drift_h} then yields
    \begin{equation} \label{eq:zcbf_like_ineq_tmp}
        -h(x)^{-1} \tilde{\mu} + h(x)^{-3} \frac{\tilde{\sigma}^2}{h(x)} \leq \alpha_3(h).
    \end{equation}
    Since $h(x) > 0$, simplifying \eqref{eq:zcbf_like_ineq_tmp} results in \eqref{eq:zcbf_like_ineq}.
    Hence, when \eqref{eq:zcbf_like_ineq} holds, $B$ is a RCBF with $\alpha_1(h) = \alpha_2(h) = h$. The result then follows from \Cref{thm:rcbf_wp1:fixed}.
\end{proof}
Notice that the modified ZCBF condition \eqref{eq:zcbf_like_ineq_tmp} has an additional \textit{correction term} $-\tilde{\sigma}^2 / h(x)$ compared to the original ZCBF condition \eqref{eq:zcbf}. Since this correction term diverges to $-\infty$ as $h(x) \downarrow 0$, the drift term $\tilde{\mu}$ must also become unbounded for the modified ZCBF condition \eqref{eq:zcbf_like_ineq} to hold. However, as we show in the next section, having $\tilde{\mu}$ become unbounded as $h(x)$ approaches zero is not sufficient to guarantee safety with probability one.

\section{Unbounded controls are insufficient for almost-sure safety}

While one might be led to believe that the ability of RCBFs to maintain safety with probability one is due to the fact that the controls can be unbounded, we show in the following section that this is not sufficient.
To better understand \textit{how} an RCBF guarantees safety, we take $h(x) = 1 / B(x)$ as in the previous section. By rearranging the Modified ZCBF constraint \eqref{eq:zcbf_like_ineq}, we obtain
\begin{equation}
    \frac{\tilde{\mu}_t}{\tilde{\sigma}_t^2} \geq h^{-1} - \tilde{\sigma}^{-2} \alpha_3( h ).
\end{equation}
In other words, as $h \downarrow 0$, $h^{-1} \to \infty$, and thus the ratio $\tilde{\mu} / \tilde{\sigma}^2$ also becomes unbounded.
However, more importantly, the ratio $\tilde{\mu}/ \tilde{\sigma}^2$ diverges \textit{faster} than $h^{-1}$. 
Some works mention that the ability of RCBFs to maintain safety with probability one is due to the fact that the controls can be unbounded \cite{wang2021safety}.
However, as we show in the following lemma, the \textit{rate} of divergence of $\tilde{\mu} / \tilde{\sigma}^2$ as $h \downarrow 0$ is the deciding factor.
\begin{lem}{}{\label[lemma]{lem:unbounded_not_enough}}
    Let the scalar stochastic process $h$ satisfy the SDE
    \begin{equation} \label{eq:h_sde}
        \dd{h}_t = \tilde{\mu}_t \dd{t} + \tilde{\sigma}_t \dd{W}_t,
    \end{equation}
    for $h \in (0, \infty)$, and suppose that the ratio $\tilde{\mu} / \tilde{\sigma}^2$ satisfies
    \begin{equation} \label{eq:lem_unbounded:ratio}
        \frac{\tilde{\mu}_t}{\tilde{\sigma}_t^2} = \gamma h^{-p}
    \end{equation}
    for any $\gamma \geq 0$ and $p > 0$. Define the stopping time $T = \inf\{t \geq 0 : X_t = 0\}$.
    Then,
    \begin{enumerate}
        \item If $p \in [0, 1)$ and $\tilde{\sigma}$ is bounded from below, then
        \begin{equation}
            \Pr(T < \infty) > 0.
        \end{equation}
        \item If $p = 1$ and $\gamma \geq 1/2$, or $p > 1$, then
        \begin{equation}
            \Pr\lft( \inf_{t \geq 0} h_t > 0 \rht) = 1 \text{ and } \Pr(T = \infty) = 1.
        \end{equation}
    \end{enumerate}
\end{lem}
Hence, if we take $\tilde{\sigma}$ constant and $\tilde{\mu}_t = h^{-1/2}$, then 
there with some nonzero probability, there exists some $t \geq 0$ such that $x_t$ exits $\interior(C)$ and $h_t=0$.
In some sense, the stochastic RCBF asks for the slowest polynomial divergence rate of $\tilde{\mu} / \tilde{\sigma}^2$ that still guarantees safety with probability one.

\section{Conclusion}
In this letter, we have revealed a flaw in the almost-sure safety guarantees of stochastic ZCBFs as constructed in \cite[Theorem 4]{clark2021control}, and provided a new proof technique to show that stochastic RCBFs are indeed safe with probability one.
The appeal of ZCBFs as a way to guarantee safety almost surely without the need for unbounded drift has attracted much attention in literature \cite{pereira2021safe,song2022generalization,vahs2023non,vahs2023risk,enwerem2023safe}. Unfortunately, these claims are not valid. Some works may already suspect this to be the case, as in \cite{wang2021safety}, which references only the safety guarantees of stochastic RCBFs but not stochastic ZCBFs.
We hope this letter can inspire the community to revisit previous claims made about stochastic ZCBFs and potentially construct new methods that build upon the new proof techniques and insights presented.

\begin{appendix}
\section{Proof of \Cref{lem:unbounded_not_enough}}
Before we begin, we first state the following two theorems, which are a direct application of Feller's Test for Explosions \cite[Proposition 5.22]{karatzas2012brownian}
\begin{thm}[Sufficient conditions for non-negativity]\label[theorem]{thm:feller}
    Let $X$ be a weak solution of the following scalar SDE
    \begin{equation}
        \dd{X}_t = \mu(X_t) \dd{t} + \sigma(X_t) \dd{W}_t.
    \end{equation}
    where $\mu$ and $\sigma$ are defined on $(0, \infty)$.
    For a fixed $c > 0$, define the scale function $s(x)$ for $X$ as
    \begin{equation} \label{eq:scale_fn_def}
        s(x) \coloneqq \int_c^x \exp\left( -2 \int_c^y \frac{\mu(z)}{\sigma(z)} \dd{z} \right) \dd{y},
    \end{equation}
    and the speed function $v(x)$ for $X$ as
    \begin{equation} \label{eq:speed_fn_def}
        v(x) \coloneqq \int_c^x s'(x) \int_c^{y} \frac{2}{s'(z) \sigma^2(z)} \dd{z} \dd{y}.
    \end{equation}
    Finally, define the stopping times
    \begin{align}
        T \coloneqq \inf\Big\{ t \geq 0: X_t \not \in (0, \infty) \Big\}
    \end{align}
    \begin{enumerate}
        \item If $s(0) = -\infty$ and $s(\infty) < \infty$,
        \begin{equation}
            \Pr\lft( \inf_{t \geq 0} X_t > 0 \rht) = 1,
        \end{equation}
        and $X_t$ is strictly positive with probability one for all time.
\item If both $s(0)$ and $s(\infty) \coloneqq \lim_{d \to \infty} s(d)$ are finite,
        \begin{equation}
            \Pr(X_T = 0) = \frac{s(\infty) - s(X_0)}{s(\infty) - s(0)} > 0.
        \end{equation}
        Moreover, if either $v(0)$ or $v(\infty)$ are finite, then $T$ is finite a.s., and
        \begin{equation}
            \Pr(T < \infty) > 0.
        \end{equation}
\item If $s(0) = -\infty$ and $s(\infty) = \infty$, then $T$ is infinite a.s., but $X_t$ gets arbitrarily close to both $0$ and $\infty$, i.e.,
        \begin{align}
            \Pr\lft(\inf_{t\geq0} X_t = 0\rht) = \Pr\lft(\sup_{t \geq 0} X_t = \infty \rht) &= 1, \\
            \Pr\lft(T = \infty\rht) &= 1.
        \end{align}
    \end{enumerate}
\end{thm}
\begin{thm}[]
\end{thm}

Next, we define the incomplete Gamma function $\Gamma(a, x)$ \cite{gautschi1998incomplete,tricomi1950sulla} and state some of its properties.
\begin{defn}{(Incomplete Gamma Function)}
    The incomplete Gamma function $\Gamma(a, x)$ is defined as
    \begin{equation}
        \Gamma(a, x) \coloneqq \int_x^\infty t^{a - 1} e^{-t} \dd{t}.
    \end{equation}
\end{defn}
\begin{lem}{}
    By definition of $\Gamma$,
    \begin{equation} \label{eq:gamma_limit}
        \lim_{x \to \infty} \Gamma(a, x) = 0.
    \end{equation}
    Moreover, for $a > 0$, $\Gamma(a, 0)$ is finite.
\end{lem}

We now proceed to prove \Cref{lem:unbounded_not_enough}.
\begin{proof}
    For conciseness, let $q \coloneqq 1 - p$ such that $-p = q - 1$, and denote $s(\infty) \coloneqq \lim_{x \to \infty} s(x)$.

    We begin by computing the scale function $s(x)$ for $h$ \eqref{eq:h_sde}.
    Using \eqref{eq:lem_unbounded:ratio}, the inner integral $\int_c^y \frac{\tilde{\mu}(z)}{\tilde{\sigma}^2(z)} \dd{z}$ can be computed as
    \begin{equation}
        \eta(y) \coloneqq \int_c^y \gamma z^{q - 1} \dd{z} = \begin{dcases} \gamma \ln(y / c), & q = 0, \\ \frac{\gamma}{q} \left( y^{q} - c^{q} \right), & q \neq 0. \end{dcases}
    \end{equation}
    We now consider into three cases.

\textbf{Case 1. $q=0$ ($p=1$)}\\
    Computing $s$ using \eqref{eq:scale_fn_def}, we get
    \begin{equation}
        s(x) = \frac{1}{2\gamma - 1} \left(c - \frac{c^{2\gamma}}{x^{2\gamma - 1}}  \right)
    \end{equation}
    Note that
    \begin{align}
        \lim_{x \downarrow 0} s(x) &= \begin{dcases}
            -\infty, &\qquad \gamma > \frac{1}{2}, \\
            -\infty, &\qquad \gamma = \frac{1}{2}, \\
            \frac{c}{2\gamma - 1}, \qquad &\gamma \in (0, 1/2).
        \end{dcases} \\
        \lim_{x \to \infty} s(x) &= \begin{dcases}
            \frac{c}{2\gamma - 1}, &\gamma > \frac{1}{2},  \\
            \infty, &\gamma = \frac{1}{2}, \\
            \infty, &\gamma \in (0, 1/2). \end{dcases}
    \end{align}
    Hence, if $\gamma \geq \frac{1}{2}$, then $\abs{s(0)} = \infty$, and the process $h$ is strictly positive with probability one by \Cref{thm:feller}.

\textbf{Case 2. $q > 0$ ($p < 1$)}\\
    In this case, using \eqref{eq:scale_fn_def} yields
    \begin{align}
        s(x) &= -d \Big( \Gamma\lft( q^{-1}, \nu(x) \rht) - \Gamma\lft( q^{-1}, \nu(c) \rht) \Big), \\
        \nu(x) &\coloneqq -\frac{2 x^q \gamma}{q}, \\
        d &\coloneqq \frac{1}{q} \exp( \nu(c) )\, (2 \gamma / q)^{1/q}.
    \end{align}
    Since $\nu(0) = 0$, and $q^{-1} > 0$, $\Gamma(q^{-1}, \nu(0))$ is finite, and hence $s(0)$ is also finite. On the other hand, since $\nu(\infty) = \infty$, using \eqref{eq:gamma_limit} gives that $s(\infty)$ is also finite. Applying \Cref{thm:feller} thus gives us that, as $t \to \infty$, $h$ hits zero with positive probability.

    We first compute $s'(x)$ as
    \begin{align}
        s'(x) &= \exp(-2 \eta(x)) = \underbrace{\exp\Big(-2 \frac{\gamma}{q} c^q \Big)}_{\coloneqq K_1} \exp\Big( \underbrace{2 \frac{\gamma}{q}}_{K_2} x^q \Big), \\
        &= K_1 \exp(K_2 x^q).
    \end{align}
    Then, the speed function $v$ can be computed as
    \begin{align}
        v(x)
        &= \int_c^x s'(y) \int_c^y \frac{2}{s'(z) \tilde{\sigma}(z)^2} \dd{z} \dd{y}, \\
        &= 2 \int_c^x K_1 e^{K_2 y^q} \int_c^y K_1^{-1} e^{-K_2 z^q} \frac{1}{\tilde{\sigma}(z)^2}\dd{z} \dd{y}, \\
        &= 2\frac{1}{\sigma(z)^2} \int_c^x \int_c^y \exp(K_2 (y^q - z^q)) \frac{1}{\tilde{\sigma}(z)^2}\dd{z} \dd{y}.
    \end{align}
    If $\tilde{\sigma}(z)$ is bounded from below, then the integrand is finite, and $v(0)$ is finite. Hence, by \Cref{thm:feller}, $\Pr(T < \infty) > 0$.

\textbf{Case 3. $q < 0$ ($p > 1$)}\\
    For clarity, define $\epsilon \coloneqq -q > 0$. Then,
    \begin{align}
        s(x)
        &= \int_c^x \exp\lft( \frac{2 \gamma}{\epsilon} \left( y^{-\epsilon} - c^{-\epsilon} \right) \rht) \dd{y} \\
        &= \exp\lft(-\frac{2}{\gamma}{\epsilon} c^{-\epsilon}\rht) \int_c^x \exp\lft( \frac{2 \gamma}{\epsilon} y^{-\epsilon} \rht) \dd{y}
    \end{align}
    We first evaluate $s(\infty)$. Since $\epsilon > 0$, the integrand $\exp(2\gamma/\epsilon y^{-\epsilon})$ is bounded below by $1$. Hence, the integral diverges, i.e.,
    \begin{equation}
        \lim_{x \to \infty} s(x) = \infty.
    \end{equation}
    On the other hand, to evaluate $s(0)$, consider a Taylor expansion of the exponential function:
    \begin{equation}
        \exp\lft( r y^{-\epsilon} \rht) = \sum_{n=0}^\infty \frac{1}{n!} \left(d y^{-\epsilon} \right)^n
    \end{equation}
    where $r \coloneqq \frac{2 \gamma}{\epsilon}$ for conciseness.
    Integrating both sides from $c$ to $x$ yields
    \begin{equation} \label{eq:integrate_series}
        \int_c^x \exp\lft( r y^{-\epsilon} \rht) = \sum_{n=0}^\infty \frac{1}{n!} r^n \frac{1}{1 - n \epsilon} \left( x^{1 - n \epsilon} - c^{1-n\epsilon} \right)
    \end{equation}
    Since $\epsilon > 0$, we have that $1 - n \epsilon < 0$ for $n$ large enough, which causes $x^{1 - n \epsilon}$ to diverge as $x \downarrow 0$. Hence, the integral \eqref{eq:integrate_series}, and thus $s$, diverges to $-\infty$ as $x \downarrow 0$, and we have that
    \begin{equation}
        \lim_{x \downarrow 0} s(x) = -\infty.
    \end{equation}
    Since $s(0) = -\infty$ and $s(\infty) = \infty$, by \Cref{thm:feller},
    \begin{equation}
        \Pr( \inf_{t \geq 0} X_t = 0) = \Pr( \sup_{t \geq 0} X_t = \infty) = \Pr( T = \infty) = 1.
    \end{equation}
\end{proof}

\end{appendix}

\bibliographystyle{plain}        \bibliography{bibliography}

\end{document}